\documentclass[11pt]{article}
\usepackage{amsmath,amssymb,amsthm,bbm,mathtools,comment}
\usepackage[shortlabels]{enumitem}
\usepackage[pdftex,colorlinks,backref=page,citecolor=blue]{hyperref}
\usepackage{tikz}
\usetikzlibrary{shapes.misc,calc,intersections,patterns,decorations.pathreplacing}
\usetikzlibrary{arrows,shapes,positioning}
\usetikzlibrary{decorations.markings}

\usepackage{cleveref}
\usepackage{xspace}

\allowdisplaybreaks

\theoremstyle{plain}
\newtheorem{theorem}{Theorem}[section]		
\newtheorem{lemma}[theorem]{Lemma}
\newtheorem*{claim}{Claim}
\newtheorem{proposition}[theorem]{Proposition}
\newtheorem{corollary}[theorem]{Corollary}

\newtheorem{definition}[theorem]{Definition}

\theoremstyle{remark}

\newcommand{\EE}{\mathbb{E}}

\DeclareMathOperator{\len}{len}
\newcommand{\eps}{\ensuremath{\varepsilon}}
\let\emptyset\varnothing

\let\originalleft\left
\let\originalright\right
\renewcommand{\left}{\mathopen{}\mathclose\bgroup\originalleft}
\renewcommand{\right}{\aftergroup\egroup\originalright}

\makeatletter
\def\imod#1{\allowbreak\mkern10mu({\operator@font mod}\,\,#1)}
\makeatother

\newcommand{\floor}[1]{\left\lfloor #1 \right \rfloor}
\newcommand{\ceil}[1]{\left\lceil #1 \right \rceil}
\newcommand{\subs}{\subseteq}
\newcommand{\sups}{\supseteq}
\newcommand{\sm}{\setminus}
\newcommand{\cplt}{\rightrightarrows} 
\newcommand{\tA}{\tilde{A}}
\newcommand{\tB}{\tilde{B}}
\newcommand{\tT}{\tilde{T}}

\newcommand{\tn}{\tilde{n}}
\newcommand{\teps}{\tilde{\eps}}
\newcommand{\intrans}{intransitive\xspace} 

\title{Powers of paths and cycles in tournaments}

\author{
Ant\'onio Gir\~ao\thanks{Institut f\"ur Informatik, Universit\"at Heidelberg, Germany. E-mail: \texttt{a.girao@informatik.uni-heidelberg.de}.}
\thanks{Supported by Deutsche Forschungsgemeinschaft (DFG, German Research Foundation) under Germany’s Excellence Strategy EXC-2181/1 - 390900948 (the Heidelberg STRUCTURES Cluster of Excellence).}
\and
D\'aniel Kor\'andi\thanks{Mathematical Institute, University of Oxford, Andrew Wiles Building, Radcliffe Observatory Quarter, Woodstock Road, Oxford, United Kingdom.   Emails: \texttt{\{korandi, scott\}@maths.ox.ac.uk}.} \thanks{Supported by SNSF Postdoc.Mobility Fellowship P400P2\_186686. }
\and
Alex Scott\footnotemark[3]
\thanks{Research supported by EPSRC grant EP/V007327/1.}
}

\date{}
\begin{document}

\maketitle

\begin{abstract}
We show that for every positive integer $k$, any tournament can be partitioned into at most $2^{ck}$ $k$-th powers of paths. This result is tight up to the exponential constant. Moreover, we prove that for every $\eps>0$ and every integer $k$, any tournament on $n\ge \eps^{-Ck}$ vertices which is $\eps$-far from being transitive contains the $k$-th power of a cycle of length $\Omega(\eps n)$; both bounds are tight up to the implied constants.
\end{abstract}

\section{Introduction}

Tournaments are complete graphs where every edge has an orientation. 
A simple exercise (\cite{R34}) shows that any tournament contains a Hamilton path, i.e., a directed path which passes through every vertex. A natural generalisation of a directed path is a $k$-th power of a path which consists of a sequence of vertices $x_1,x_2,\dots,x_n$ with the property that $x_i \rightarrow x_j$ for every $1\le i<j \le i+k\le n$. When $k\ge n-1$, this $k$-th power of a path is an $n$-vertex transitive tournament.

In \cite{Y21}, Yuster investigated the problem of estimating the minimum over all $n$-vertex tournaments of the maximum length of a power of a path. He showed that any tournament on $n$ vertices must contain a power of a directed path on at least $n^{0.295}$ vertices. Confirming a conjecture of Yuster, Dragani\'c et al.~\cite{DDF+} showed that for every $k$ there always exists a $k$-th power of a path of linear order.

\begin{theorem}[Dragani\'c et al.~\cite{DDF+}] \label{thm:tournamentpath}
For every positive integer $k$, any tournament on $n$ vertices contains the $k$-th power of a path of length $n/2^{(2+o(1))k}$.
\end{theorem}

More precisely, for every positive integer $k$, every tournament on $n$ vertices contains a $k$-th power of a path on at least $kn/2^{4k+6}$ vertices. Moreover, it was shown in \cite{DDF+} that this is tight up to the constant in the exponent. 
One might be tempted to ask whether the same phenomenon holds for every acyclic digraph with bounded maximum degree. In other words, is it the case that, for any positive integer $d$, there is a constant $C(d)>0$ such that for any acyclic digraph $D$ on $n$ vertices with maximum degree $d$, every tournament on at least $C(d)n$ vertices contains a copy of $D$? 
Very recently Fox, He and Widgerson~\cite{F21} gave a negative answer to this question. Indeed, they showed that for all $\Delta\ge2$ and every sufficiently large $n$, there is an acyclic digraph $D$ on $n$ vertices with maximum degree $\Delta$ for which there are tournaments on at least $n^{\Omega(\Delta^{2/3-o(1)})}$ vertices that do not contain any copy of $D$. 

In this paper, we build some tools for finding powers of paths and cycles in tournaments and use them to prove two further results. We first extend \Cref{thm:tournamentpath}, by showing that every tournament can be partitioned into $2^{O(k)}$ tournaments each of which contains the $k$-th power of a Hamilton path.  We then consider powers of cycles, showing that if a tournament is far from transitive then it must contain the $k$th power of a long cycle; we give bounds that are essentially tight up to the implicit constants.  These results are discussed in the next two subsections.

\subsection{Partitions into powers of paths}

A common theme in combinatorics is  the problem of partitioning the vertex set of a graph into a bounded number of pieces, each satisfying certain properties.
A famous problem in this area is Lehel's Conjecture, which states that any graph has a vertex partition into two parts, where one forms a cycle and the other forms an anticycle (a cycle in the complement);
equivalently, any $2$-edge-coloured graph has a vertex bipartition into two monochromatic cycles of distinct colours
This was confirmed for large enough graphs by \L uczak, R\"{o}dl, and Szemer\'edi~\cite{LRS98} and for all graphs by Bessy and Thomass\'e~\cite{BT10}. 

Similar questions arise for colourings with more colours.  An influential result of Erd\H{o}s, Gy\'arf\'as and Pyber~\cite{EGP91} shows that any $r$-edge-coloured graph can be partitioned into at most $O(r^2\log r)$ monochromatic copies of a cycle. 
Recently Bustamante, Corsten, Frankl, Pokrovskiy and Skokan~\cite{BCFPS20} extended this to powers of cycles, proving that that for all natural numbers $k$ and $r$, the vertices of every $r$-edge-coloured complete graph can be partitioned into a bounded number of $k$-th powers of cycles. (We refer the reader to the survey of Gy\'arf\'as~\cite{G16}, for many other problems dealing with partitions and covers of finitely edge-coloured graphs.)

In light of these results it is very natural to ask whether every tournament has a finite partition of the vertex set into $k$-th powers of paths. In our first result, we answer this question.

\begin{theorem} \label{thm:partitioning}
Every $n$-vertex tournament $T$ can be partitioned into at most $2^{10^5k}$ vertex-disjoint $k$-th powers of directed paths.
\end{theorem}
We remark that the bound above is essentially tight, up to the implied constant. To see this, let $T_k$ be a tournament on $2^{k/2}$ vertices which does not contain a transitive tournament on $k$ vertices, and let $T$ be a tournament consisting of the disjoint union of $n/2^{k/2}$ copies of $T_k$, where all edges between the copies are oriented from left to right. It is easy to see that $T$ does not contain a $k$-th power of a path of length greater than $kn/2^{k/2}$. 

\subsection{Powers of cycles in $\eps$-intransitive tournaments}

Generalizing the previous results about path powers to cycle powers is not possible. For example, the transitive tournament does not contain any directed cycle at all. One can only say something about cycles in tournament by making further structural assumptions.

An old result of Bollob\'as and H\"aggkvist \cite{BH90} says that for every $k$ and $\eps>0$, any tournament on sufficiently many vertices with minimum semi-degree at least $(1/4+\eps)n$ contains a $k$-th power of a Hamilton cycle and this is tight up to the $o(1)$ error term. Recently, Draganic, Munh\'a Correia, and Sudakov \cite{DMS} were able to find an almost tight bound for the error term. 

We complement these results by showing that much milder assumptions are sufficient for the existence of long (linear-length) $k$-th powers of cycles in tournaments.

We say that an $n$-vertex tournament is \emph{\eps-\intrans} if no matter how we order its vertices, there are always at least $\eps n^2$ backward edges. 
This is a way to measure how far the tournament is from being transitive. Note that a tournament cannot be $\gamma$-\intrans for $\gamma> 1/4$, because for any vertex ordering $\tau$, either $\tau$ or its reverse induces fewer than $n^2/4$ backward edges.
This definition turns out to be quite important and is a sort equivalent of edge density for tournaments. Indeed, as a by-product of a result of Chung and Graham~\cite{CG91}, it follows that for any tournament $H$, there is $\eps(H)\geq 0$ such that every sufficiently large tournament $T$ which is $(\eps(H)+o(1))$-intransitive must contain a copy of $H$. Fox and Sudakov~\cite{FS08} showed that $\eps(F)=0$ works for any transitive blow-up $F$ of a directed triangle. 

We prove the following result. 

\begin{theorem} \label{thm:cycles}
Let $0<\eps<1/4$, then every $\eps$-\intrans tournament on $n\ge \eps^{-10^5 k}$ vertices contains the $k$-th power of a cycle of length at least $\eps n/1500$.
\end{theorem}

Furthermore, in \Cref{seccycle}, we show both bounds are essentially tight.

The rest of the paper is organized as follows.  In the next section, we develop some machinery that we shall need for our proofs. \Cref{thm:partitioning}, on partitioning into powers of cycles, is proved in \Cref{secpartition}; and \Cref{thm:cycles}, on finding powers of long cycles, is proved in \Cref{seccycle}.

We note that we have not tried to optimize absolute constants in our results.

\section{Tools} \label{sec:tools}

A tournament is transitive if its vertices can be ordered so that every edge is oriented from its smaller endpoint to its larger endpoint. Note that every subtournament of a transitive tournament is also transitive.

We use $[n]$ to denote the set of integers $\{1,\dots,n\}$.
$T$ is a tournament. For two vertex sets $A,B$ in $T$, $e(A,B)$ denotes the number of edges from $A$ to $B$, and $d(A,B) = \frac{e(A,B)}{|A||B|}$ denotes the density of such edges. Of course, $d(A,B)=1$ if and only if $A$ and $B$ are disjoint and there is an edge from every vertex in $A$ to every vertex in $B$. We will denote this by $A\cplt B$.

For convenience, our notion of cycles will include singleton vertices and edges as degenerate cases on one or two vertices. We will also drop floor and ceiling signs when they are not essential.

\subsection{Extremal lemmas}

Let us start with two well-known facts that we will need for the proof. The first is a basic property of tournaments.

\begin{proposition} \label{lem:transitive}
Every tournament on at least $2^k$ vertices contains a $k$-vertex transitive subtournament.
\end{proposition}

We will also need the following simple observation about edge densities.

\begin{proposition} \label{lem:mindegrees}
Let $G=(A\cup B, E)$ be a bipartite graph with $|E| = \beta|A||B|$ edges for some $0<\beta\le 1$. Then for every $\eps\ge 0$,  $A$ contains at least $(\beta-\eps)|A|$ vertices of degree at least $\eps|B|$.
\end{proposition}

Next, we formulate a special case of the K\H{o}v\'ari-S\'os-Tur\'an theorem, which will be a key tool in our arguments. We give a short proof for completeness.

\begin{lemma} \label{lem:mindegturan}
Let $G=(A\cup B, E)$ be a bipartite graph such that for some $0<\beta\le 1/2$, every vertex in $A$ has at least $\beta|B|$ neighbours in $B$. If $|A|\ge k/\beta$, then $A$ contains a subset $X$ of size $k$ with at least $\beta^{4k}|B|$ common neighbours in $B$.
\end{lemma}
\begin{proof}
We may assume that $|A| = \ceil{k/\beta}$. A vertex $v\in B$ sees $\binom{d(v)}{k}$ different $k$-subsets of $A$ in its neighbourhood. This gives a total of 
\[ \sum_{v\in B} \binom{d(v)}{k} \ge |B| \binom{\sum d(v) / |B|}{k} \ge |B| \binom{\beta |A|}{k} \ge |B| \]
$k$-sets over all vertices of $B$, where we used Jensen's inequality, and that $\sum d(v) \ge \beta|A||B| \ge k|B|$. But there are only $\binom{|A|}{k} \le \binom{2k/\beta}{k}\le (\frac{2k/\beta}{k/e})^k \le \beta^{-4k}$ different $k$-sets in $A$, so one of them must have at least $\beta^{4k}|B|$ common neighbours in $|B|$.
\end{proof}

We will also need the following strengthening for tournaments, which comes as a simple application of the dependent random choice method.
\begin{lemma} \label{lem:transturan}
Let $A,B$ be disjoint sets in a tournament. If $d(A,B) \ge \beta$ for some $0<\beta \le 1/2$, and $|A|,|B|\ge \beta^{-5k}$, then there are subsets $X\subs A$ and $Y\subs B$ of size $|X|\ge \beta^{4k}|A|$ and $|Y|=k$ such that $Y$ induces a transitive tournament, and $X\cplt Y$.
\end{lemma}
\begin{proof}
Let $S$ be a random set of $s=\floor{\frac{1}{2}\log_{1/\beta} |B|}$ independently and uniformly sampled vertices in $A$, and let $T\subs B$ be the set of common outneighbours of $S$. Then $\EE[|T|]=\sum_{v\in B} d(A,v)^s \ge \beta^s |B| \ge |B|^{1/2}$ by Jensen's inequality. Let $Z$ be the number of $k$-subsets in $T$ with fewer than $\beta^{4k}|A|$ common inneighbours in $A$. The probability that a given $k$-subset $Q\subseteq B$ with $\gamma|A|$ common inneighbours in $A$ satisfies $Q\subs T$ is $\gamma^s$, so we have $\EE[Z] = \binom{|B|}{k}\beta^{4k\cdot s} \le (\beta^{4s}|B|)^k \le 1$, as $4s\geq \log_{1/\beta}(|B|)$. Hence $\EE[|T|-Z]\ge |B|^{1/2} - 1$. Let us fix a random sample where $|T|-Z\ge |B|^{1/2}-1$.

By deleting a vertex of each $k$-subset counted by $Z$ from the set $T$, we obtain a subset $W\subs T$ of size at least $|B|^{1/2}-1$ such that all $k$-subsets in $W$ have at least $\beta^{4k}|A|$ common inneighbours. As $|W|\ge 2^{2k}$, we can use \Cref{lem:transitive} to find a $k$-set $Y\subs W$ that induces a transitive subtournament. We can choose $X$ to be the common inneighbourhood of $Y$.
\end{proof}

\begin{corollary} \label{cor:transturan}
Let $A,B$ be disjoint sets in a tournament. If $d(A,B) \ge \beta$ for some $0<\beta \le 1/2$, and $|A|,|B|\ge \beta^{-5k}$, then there are sets $X\subs A$ and $Y\subs B$ that induce transitive tournaments of size $k$ and satisfy $X\cplt Y$.
\end{corollary}
\begin{proof}
This is immediate from \Cref{lem:transturan} noting $\beta^{4k}|A| \ge 2^k$ and applying \Cref{lem:transitive}.
\end{proof}

Our final tool in this section describes a sufficient condition when we can repeatedly apply the previous lemmas to construct a sequence of transitive tournaments.

\begin{lemma} \label{lem:transsequence}
Let $A_1,\dots, A_t$ be disjoint vertex sets of size at least $2^{10k}$ in a tournament, and suppose there is no $i\in [t-1]$ and sets $B\subs A_i$ and $B'\subs A_{i+1}$ such that $B$ and $B'$ induce transitive tournaments of size $k$, and $B'\cplt B$.
Then there are sets $X_i\subs A_i$ of size $k$ such that each $X_i$ induces a transitive subtournament, and $X_1\cplt \cdots \cplt X_t$.
\end{lemma}
\begin{proof}
We proceed by induction on $t$ under the weaker assumption that $|A_t| \ge 2^{6k}$. The $t=1$ case easily follows from \Cref{lem:transitive}, so we assume $t\ge 2$.

If $d(A_t,A_{t-1})\ge 1/2$, then we can apply \Cref{cor:transturan} to find $k$-sets $B\subs A_{t-1}$ and $B'\subs A_t$ that induce transitive tournaments and satisfy $B'\cplt B$, contradicting our assumption.

So $d(A_t,A_{t-1})\le 1/2$, i.e., $d(A_{t-1},A_t)\ge 1/2$, and we can apply \Cref{lem:transturan} to find a $k$-subset $X_t\subs A_t$ that induces a transitive tournament, and another subset $A'_{t-1} \subs A_{t-1}$ of size $|A'_{t-1}|\ge |A_{t-1}|/2^{4k} \ge 2^{6k}$ such that $A'_{t-1}\cplt X_t$. Applying the induction hypothesis to the sets $A_1,\dots, A_{t-2},A'_{t-1}$ yields the result.
\end{proof}

\subsection{Median orderings}

A \emph{median ordering} of a tournament $T$ is an ordering $v_1 \prec \dots \prec v_n$ of the vertices that maximizes the number of forward edges, i.e., edges of the form $v_iv_j$ with $i<j$. Studying such orderings have been very helpful in understanding the structure of tournaments. An interval of vertices with respect to a median ordering is a sequence $V[i,j] = \{v_i,\dots, v_j\}$ for some $i\le j$. For two vertex subsets $X,Y$ we write $X\prec Y$ to denote that $x\prec y$ for any $x\in X$ and $y\in Y$.

\begin{lemma} \label{lem:mediandegrees}
Suppose that in a median ordering of some tournament $T$, an interval is split into subintervals $A_0\prec \dots \prec A_t$ of size $m$ each. Then every vertex $v\in A_0$ has at least $\frac{t-2}{2} m$ outneighbours in $A=A_1\cup \dots \cup A_{t-1}$, and every vertex in $v\in A_t$ has at least $\frac{t-2}{2} m$ inneighbours in $A$. 
\end{lemma}
\begin{proof}
Note that $|A| = (t-1)m$, so if some $v\in A_0$ has fewer than $\frac{t-2}{2} m$ outneighbours in $A$, then $v$ has at least $m$ fewer outneighbours in $A$ than inneighbours. As $v$ has only $m-1$ neighbours in $A_0$, moving $v$ to the end of the interval $A_0 \cup A$ (so that $A_{t-1}\prec v \prec A_t$) is guaranteed to increase the number of forward edges. This contradicts our assumption on the ordering. The statement about $A_t$ can be proved analogously.
\end{proof}

Using this lemma, we can obtain an ordered variant of \Cref{lem:mindegturan}. It will be helpful to allow a set $F$ of `forbidden' vertices. 

\begin{lemma} \label{lem:medianturan}
Suppose that in a median ordering of some tournament $T$, an interval $A$ is split into three subintervals $A_0\cup A_1 \cup A_2$ of size $m$ each, and let $F\subs A$ be a set of at most $m/4$ forbidden vertices. Then every set $A'_0\subs A_0$ of size $8k$ contains a subset $X$ of size $k$ such that for some $i\in[2]$, $X$ has at least $m/2^{12k}$ common outneighbours in $A_i\sm F$.
\end{lemma}
\begin{proof}
By \Cref{lem:mediandegrees}, every vertex in $A'_0$ has at least $m/2$ outneighbours in $A_1\cup A_2$, at least $m/4$ of which lie in $(A_1 \cup A_2)\sm F$. Applying \Cref{lem:mindegturan} with $\beta=1/8$, we get a set $X\subs A'_0$ of size $k$ with at least $2m/2^{12k}$ common outneighbours in $(A_1 \cup A_2)\sm F$. Of course, at least $m/2^{12k}$ of these common outneighbours must lie in the same set $A_i\sm F$ for some $i\in [2]$.
\end{proof}

We can now easily deduce an ordered variant of \Cref{lem:transsequence}. In this case there is no need for the assumption on backward edges: the median ordering provides all the structure we need.
\begin{lemma} \label{lem:mediansequence}
Let $k>0$, and suppose that in a median ordering of some tournament $T$, an interval is split into subintervals $A_1\prec \dots \prec A_t$ of size $m\ge 2^{20k}$ each.
Suppose we have a set $F$ of vertices $F$ such that $|F\cap A_i| \le m/8$ for every $i\in[t]$, and $X\subs A_1\sm F$ is a set of size $8k$ that induces a transitive tournament.

Then there are sets $X_1 \cplt \cdots \cplt X_s$ such that $X_1\subs X$, and each $X_i$ induces a transitive tournament of size $k$ in $A_{j_i}\sm F$, where the indices $1=j_1<\dots<j_s$ satisfy $j_{i+1}\in \{j_i+1, j_i+2\}$ for every $i$, and $j_s\in \{t-1,t\}$.
\end{lemma}
\begin{proof}
Set $j_1=1$ and $X'_1=X$. We repeat the following step for every $i=1,2,\dots$ as long as $j_i<t-1$.

Applying \Cref{lem:medianturan} to the interval $A_{j_i}\cup A_{j_i+1}\cup A_{j_i+2}$ with $A'_0=X'_i$ and forbidden vertices $(F\cap A_{j_i+1})\cup (F\cap A_{j_i+2})$ gives a $k$-set $X_i\subs X'_i$ with at least $m/2^{12k} \ge 2^{8k}$ common outneighbours in $A_{j_{i+1}}\sm F$ for some $j_{i+1}\in \{j_i+1,j_i+2\}$. By \Cref{lem:transitive}, we can find a subset $X'_{i+1}$ of $8k$ common outneighbours that induce a transitive subtournament.

This process stops with some $j_s\in \{t-1,t\}$, and we can then choose any $k$-subset $X_s\subs X'_s$ so that $X_1,\dots,X_s$ satisfy the statement.
\end{proof}

The next lemma is a key component of our arguments, and perhaps the most technical result in the paper. It says that if two sets are far enough apart in a median ordering, then we can connect them with a blowup of a path.

\begin{lemma} \label{lem:medtranssequence}
Suppose that in a median ordering of some tournament $T$, an interval is split into subintervals $A_0\prec \dots \prec A_t$ of size $m\ge 100\cdot 2^{40400k}$ each, where $t\ge 50$. Let $A'_0 \subs A_0$ and $A'_t \subs A_t$ be subsets of size at least $2^{4001k}$, and let $F$ be a set of at most $m/2$ vertices in $A=A_1\cup\dots \cup A_{t-1}$. Then there is $s\le 3$ and disjoint sets $X_0\subs A'_0$ and $X_1,\dots, X_{s-1}\subs A\sm F$ and $X_s\subs A'_t$, such that $X_0 \cplt \cdots \cplt X_s$, and each $X_i$ induces a transitive tournament of size $k$.
\end{lemma}
\begin{proof}
Let $\eps = 1/100$, and define $A^I\subs A\sm F$ as the set of vertices with at least $\eps |A'_0|$ inneighbours in $A'_0$, and $A^O\subs A\sm F$ as the set of vertices with at least $\eps |A'_t|$ outneighbours in $A'_t$. We claim that $|A^I|, |A^O| \ge (\frac{t-3}{2}-(t-1)\eps)m$. Indeed, there are at least $\frac{t-2}{2}m|A'_0| =  \frac{t-2}{2t-2}|A||A'_0|$ edges from $A'_0$ to $A$ by \Cref{lem:mediandegrees}, so \Cref{lem:mindegrees} applied to the bipartite graph induced by these edges with $\beta = \frac{t-2}{2t-2}$ gives $(\frac{t-2}{2}-(t-1)\eps)m$ vertices in $A$ with at least $\eps|A'_0|$ inneighbours in $A'_0$. Excluding the vertices of $F$ yields the lower bound on $|A^I|$. The bound on $|A^O|$ is analogous.

If $A^I$ and $A^O$ share at least $2^{k/\eps^2}$ vertices, then we are done with $s=2$ as follows. By \Cref{lem:transitive}, there is a set $Y\subs A^I\cap A^O$ of size $k/\eps^2$ that induces a transitive tournament. Then every vertex of $Y$ has at least $\eps |A'_t|$ outneighbours in $A'_t$, so we can apply \Cref{lem:mindegturan} to the bipartite graph of the edges from $Y$ to $A'_t$ to get sets $Y'\subs Y$ and $\tA_t \subs A'_t$ such that $|Y'|\ge k/\eps$ and $|\tA_t| \ge \eps^{4k/\eps}|A'_t|\ge |A'_t|/2^{4000k} \ge 2^{k}$ with $Y'\cplt \tA$. Once again, every vertex of $Y'$ has at least $\eps |A'_0|$ inneighbours in $A'_0$, so \Cref{lem:mindegturan} gives $Y''\subs Y'$ and $\tA_0\subs A'_0$ such that $|Y''|\ge k$ and $|\tA_0|\ge \eps^{4k}|A'_0| \ge 2^{k}$ with $\tA_0 \cplt Y''$. We can then choose $X_0$ and $X_2$ to be $k$-subsets of $\tA_0$ and $\tA_t$ that induce transitive tournaments, and $X_1$ to be $Y''$.

So we may assume that $|A^I\cap A^O| < 2^{k/\eps^2} < \eps m$. Let $\tA^I = A^I \sm A^O$. Then $\tA^I$ is disjoint from $A^O$ and has size at least $(\frac{t-3}{2}-t\eps)m$. Also, $|\tA^I \cup A^O| > (t-3)m - (2t-1)\eps m$, so the set $A^X = A \sm (A^I \cup A^O)$ of leftover vertices (including $F$) has size at most $2m + (2t-1)\eps m$.

If $d(\tA^I,A^O) > \eps$, then we can conclude the argument with $s=3$ similarly to the previous case: As $|\tA^I|, |A^O|\ge m\ge \eps^{-5k/\eps}$, we can apply \Cref{cor:transturan} to find sets $Y\subs \tA^I$ and $Z\subs A^O$ that induce transitive tournaments of size $k/\eps$ and satisfy $Y\cplt Z$. Now every vertex of $Y$ has at least $\eps |A'_0|$ inneighbours in $A'_0$ and every vertex of $Z$ has at least $\eps |A'_t|$ outneighbours in $A'_t$, so two independent applications of \Cref{lem:mindegturan} gives $k$-sets $Y'\subs Y$ and $Z'\subs Z$ as well as sets $\tA_0\subs A'_0$ and $\tA_t\subs A'_t$ of size $|\tA_0| \ge \eps^{4k}|A'_0| \ge 2^k$ and $|\tA_t| \ge \eps^{4k}|A'_t| \ge 2^k$ such that $\tA_0\cplt Y'\cplt Z'\cplt \tA_t$. We then choose $X_0$ and $X_3$ to be transitive $k$-subsets of $\tA_0$ and $\tA_t$, respectively, and set $X_1=Y'$ and $X_2=Z'$.

So let us also assume that $d(\tA^I,A^O) \le \eps$. We will show that we could not have started with a median ordering in this case. Let $B= A_1\cup \dots \cup A_{\floor{(t-1)/2}}$ and $C=A_{t-\floor{(t-1)/2}}\cup \dots \cup A_{t-1}$, and let us first bound the size of $B^I = B\cap A^I$ and $C^O = C\cap A^O$. By \Cref{lem:mediandegrees}, there are at least $\frac{t-6}{4}|A'_0|m$ edges from $A'_0$ to $B$ and $\frac{t-6}{4}|A'_t|m$ edges from $C$ to $A'_t$, so \Cref{lem:mindegrees} implies $|B^I|,|C^I| \ge (\frac{t-6}{4}-(t-1)\eps)m$. We then also see that $\tB^I = B\cap \tA^I$ has size at least $(\frac{t-6}{4}-t\eps)m$.

Let us now consider the ordering on the vertices of the tournament that moves all vertices in $\tA^I$ to the end of the interval $A$, without affecting the relative order of vertices in any other way. So $A'_0 \prec (A^O\cup A^X) \prec \tA^I \prec A'_t$ in this new ordering. With this reordering, we may lose up to $|\tA^I||A^X|$ forward edges between $\tA^I$ and $A^X$, and up to $\eps |\tA^I||A^O|$ between $\tA^I$ and $A^O$, but we will surely gain at least $|\tB^I||C^O|-\eps |\tA^I||A^O|$ forward edges between $\tB^I$ and $C^O$.
With $t\ge 50$ and $\eps=1/100$, we can bound the terms as follows:
\begin{align*}
 |\tB^I||C^O| &\ge \left(\frac{t-6}{4}-t\eps\right)\left(\frac{t-6}{4}-(t-1)\eps\right)m^2 \ge \frac{(t-2)^2}{25} \cdot m^2 \\
 |\tA^I||A^O| &\le \left(\frac{t-2}{2} - t\eps\right) \left(\frac{t-2}{2} - (t-1)\eps\right)m^2 \le \frac{(t-2)^2}{4} \cdot m^2 \\
 |\tA^I||A^X| &\le \left(\frac{t-2}{2} - t\eps\right)\left(2 + (2t-1)\eps\right)m^2 \le \frac{t-2}{2}\cdot \frac{t-2}{15} \cdot m^2.
\end{align*}
This means that the new ordering has at least
\[
 |\tB^I||C^O|- 2\eps |\tA^I||A^O| - |\tA^I||A^X| \ge \left( \frac{1}{25} - \frac{1}{200} - \frac{1}{30} \right) (t-2)^2m^2 \ge \frac{(t-2)^2m^2}{600}
\]
more forward edges than the median ordering we started with, which is a contradiction.
\end{proof}

It will be more convenient for us to apply the previous lemma via the following statement.

\begin{corollary} \label{cor:medtranssequence}
Suppose that in a median ordering of some tournament $T$, an interval is split into subintervals $A_0\prec \dots \prec A_t$ of size $m\ge 100\cdot 2^{40400k}$ each, where $t\ge 60$. Let $X \subs A_0$ and $X' \subs A_t$ be $4k$-subsets that induce transitive tournaments, and let $F$ be a set of at most $m/2$ forbidden vertices in $A=A_1\cup\dots \cup A_{t-1}$. Then there is $s\le 5$ and disjoint sets $X_0\subs X$ and $X_1,\dots, X_{s-1}\subs A\sm F$ and $X_s\subs X'$, such that $X_0 \cplt \cdots \cplt X_s$, and each $X_i$ induces a transitive tournament of size $k$.
\end{corollary}
\begin{proof}
By \Cref{lem:mediandegrees}, every vertex in $X$ has at least $3m/2$ outneighbours in $A_1\cup A_2\cup A_3\cup A_4$, at least $m$ of which are not in $F$. So we can apply \Cref{lem:mindegturan} with $A=X$, $B= (A_1\cup A_2\cup A_3\cup A_4)\sm F$ and $\beta=1/4$ to find a $k$-subset $X_0\subs X$ with at least $3m/2^{8k}$ common outneighbours in $B$. At least $m/2^{10k}\ge 2^{20001k}$ of these are in the same $A_i$ with $i\le 4$, let us denote them by $A'_i$.

The same argument can be applied to $X'$ from the other direction, so we similarly get a $k$-set $X_t\subs X'$ with and another set $A'_{i'}\subs A_{i'}\sm F$ for $i'\ge t-4$ such that $|A'_{i'}|\ge 2^{2000k}$ and $A'_{i'}\cplt X_t$.

But then we can apply \Cref{lem:medtranssequence} to the interval $A_i\cup\dots\cup A_{i'}$ to find transitive $k$-sets $Y_0\cplt\cdots\cplt Y_{s'}$ with $s'\le 3$ such that $Y_0\subs A'_i$, $Y_1,\dots, Y_{s-1}\subs (A_{i+1}\cup\dots\cup A_{i'-1})\sm F$, and $Y_{s'}\subs A'_{i'}$. Relabeling $X_0,Y_0,\dots,Y_{s'},X_t$ as $X_0,\dots, X_s$ with $s=s'+2$ yields the sets we were looking for.
\end{proof}

\subsection{Absorbers}

\begin{definition}
We say that a tournament $H$ is a \emph{$k$-absorber} if there is an $r'$-set $Q=\{q_1,\dots,q_{r'}\}$ and a partition $V(H)=S_0\cup \dots \cup S_r\cup Q$ such that
\begin{enumerate}[(i)]
    \item $|Q|=r'=2^{10k}$ and $r>r'$.
    \item $S_0\cplt S_1\cplt \cdots \cplt S_r \cplt S_0$, and each $S_i$ induces a transitive tournament of size $2k$.
    \item $S_0 \cplt Q\cplt S_{r'+1}$, and $S_i\cplt q_i \cplt S_{i+1}$ for $i\in [r']$. 
\end{enumerate}
We will refer to the set $Q=Q(H)$ as the \emph{absorbing part} of $H$.
\end{definition}

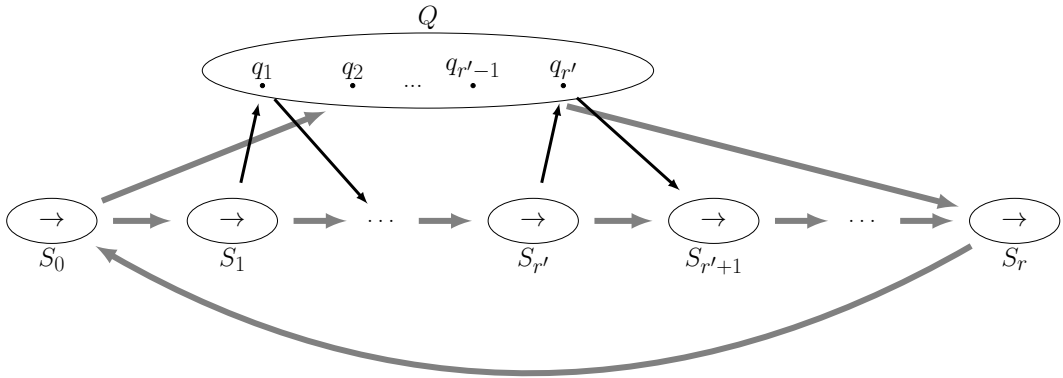
\begin{figure}[h]
    \centering
    \scalebox{0.4}{
        \begin{tikzpicture}
    \newcommand{\cVertexSize}{0.5mm}
    \newcommand{\cFatArrowColor}{gray}              
    \newcommand{\cThinArrowColor}{black}              
    \newcommand{\cMiddleBlobsVerticalOffset}{6}     
    \newcommand{\drawVertex}[4]{
        \node[circle, draw=black, fill=black, inner sep=\cVertexSize, label=\Huge{#1}](#2) at (#3, #4) {};
    }
    \newcommand{\transitiveBlob}[4]{
    \node[ellipse, minimum height=1.5cm,minimum width=3cm, draw=black, label=below:\Huge{#1}](#2) at (#3, #4) {\Huge{$\rightarrow$}};
    }
    \node[ellipse, minimum height=1.5cm,minimum width=3cm, draw=black, label=below:\Huge{$S_0$}](blobS) at (-12, 0) {\Huge{$\rightarrow$}};
    
    \node[ellipse, minimum height=1.5cm,minimum width=3cm, draw=black, label=below:\Huge{$S_1$}](blobT) at (-6,0) {\Huge{$\rightarrow$}};

 \node[ellipse, minimum height=1.5cm,minimum width=3cm, draw=black, label=below:\Huge{$S_{r'}$}](blobY) at (4,0) {\Huge{$\rightarrow$}};

 \node[ellipse, minimum height=1.5cm,minimum width=3cm, draw=black, label=below:\Huge{$S_{r'+1}$}](blobL) at (10,0) {\Huge{$\rightarrow$}};

 \node(dotsFirst) at ($(blobT)!.5!(blobY)$) {\Huge{\ldots}};
 
     \node[ellipse, minimum height=1.5cm,minimum width=3cm, draw=black, label=below:\Huge{$S_{r}$}](blobM) at (20,0) {\Huge{$\rightarrow$}};

      \node[ellipse, draw=black, minimum height=2.5cm,minimum width=15cm, label=\Huge{$Q$}](blobQ) at (0.5,5) {};
    
  \node(dotsSecond) at ($(blobL)!.5!(blobM)$) {\Huge{\ldots}};

    \begin{scope}[->, >=latex, color=\cFatArrowColor, line width=2mm, shorten >=15pt,shorten <=15pt]
        \draw (blobS) edge (blobT);
        \draw (blobT) edge (dotsFirst);
        \draw (dotsFirst) edge (blobY);
        \draw (dotsSecond) edge (blobM);
        \draw (blobL) edge (dotsSecond);
        \draw (blobS) edge (blobQ);
        \draw (blobQ) edge (blobM);
        \draw (blobY) edge (blobL);
        \draw (blobM) edge[bend left] (blobS);
   
    \end{scope}

    
    \drawVertex{$q_1$}{q1}{ - 5}{4.5}
    \drawVertex{$q_2$}{q2}{ -2}{4.5}
    \drawVertex{$q_{r'-1}$}{q_{r'-1}}{ 2}{4.5}
    \drawVertex{$q_{r'}$}{q_{r'}}{5}{4.5}
    \node at (0, 4.5) {\Huge{...}};
    
    \begin{scope}[->, >=latex, line width=1mm, shorten >=15pt,shorten <=15pt]
        \draw (blobT) edge (q1);
        \draw (q1) edge (dotsFirst);
        \draw (blobY) edge (q_{r'});
        \draw (q_{r'}) edge (blobL);
    \end{scope}

\end{tikzpicture}
    }
    \label{fig:construction_2_5_k}
\end{figure}

It is easy to see that every $k$-absorber $H$ contains the $k$-th power of a Hamilton cycle. In fact, its structure is more robust.
The key property of these absorbers is captured by the following statement.

\begin{proposition} \label{lem:absorber1cover}
Let $H$ be a $k$-absorber, and $X,Y\subs Q(H)$ be two vertex sets of size $2k$ in the absorbing part of $H$ that induce transitive subtournaments. Then $H$ contains the $k$-th power of a Hamilton path whose first $k$ vertices are in $Y$ and last $k$ vertices are in $X$.
\end{proposition}
\begin{proof}
Let $Y_0 \subs Y$ and $X_0 \subs X\sm Y_0$ be arbitrary disjoint subsets of size $k$. We will cover the vertices of $H$ by the $k$-th power of a path whose first $k$ vertices are the ones in $Y_0$ and last $k$ vertices are the ones in $X_0$.

Let us split each vertex set $S_i$ of the $k$-absorber $H$ arbitrarily into two $k$-sets $S^1_i$ and $S^2_i$. The backbone of the path power is given by the relations
\[ Y_0 \cplt S^1_{r'} \cplt \cdots \cplt S^1_r \cplt S^1_0 \cplt S_1 \cplt \cdots \cplt S_{r'-1} \cplt S^2_{r'} \cplt \cdots \cplt S^2_r \cplt S^2_0 \cplt X_0.  \]
As each of these sets induces a transitive subtournament of size at least $k$, we can combine them into the $k$-th power of a path touching all vertices in the above order. Finally, we can insert any leftover vertex $q_i\in Q \sm (X_0\cup Y_0)$ between $S_{i-1}$ and $S_i$ (or between $S^1_0$ and $S_1$ when $i=1$, and between $S_{r'-1}$ and $S^2_{r'}$ when $i=r'$) by the assumptions on $k$-absorbers.
\end{proof}

\begin{lemma} \label{lem:absorbers}
Let $H_1,\dots, H_s$ be vertex-disjoint $k$-absorbers in a tournament $T$. Then $T$ contains the $k$-th power of a directed path with vertex set $V(H_1)\cup\dots\cup V(H_s)$.
\end{lemma}
\begin{proof}
Consider the auxiliary tournament on vertex set $[s]$, where $ij$ is an edge if $d(Q(H_i),Q(H_j))\ge 1/2$ (keeping only one of $ij$ and $ji$ if $d(Q(H_i),Q(H_j))= 1/2$). Like every tournament, this must contain a Hamilton path, so we may assume that $1,\dots,s$ is a directed path.

Now for every $i\in [s-1]$, we can apply \Cref{cor:transturan} to $Q(H_i)$ and $Q(H_{i+1})$ with $\beta = 1/2$ to get $2k$-sets $X_i \subs Q(H_i)$ and $Y_{i+1} \subs Q(H_{i+1})$ that induce transitive tournaments and satisfy $X_i \cplt Y_{i+1}$. Let us set $Y_1 = X_1$ and $X_s = Y_s$. Then by \Cref{lem:absorber1cover} each $H_i$ contains the $k$-th power of some spanning path $P_i$ that starts with $k$ vertices in $Y_i$ and ends with $k$ vertices in $X_i$. As $X_i \cplt Y_{i+1}$, the concatenation of these paths satisfies our requirements.
\end{proof}

The next lemma is our tool for finding $k$-absorbers in the tournament. Its proof heavily uses the underlying median ordering, but at this point, it is simply a combination of previously established lemmas.

\begin{lemma} \label{lem:absorbfree}
Suppose that in a median ordering of some tournament $T$, an interval is split into subintervals $A_0\prec \dots \prec A_t$ of size $m\ge 2^{81000k}$ each, where $t\ge 80$. If there are sets $X_0\subs A_0$ and $X_t\subs A_t$ of size $|X_0|,|X_t|\ge 8k$ that both induce transitive subtournaments and $X_t\cplt X_0$, then $T$ contains a $k$-absorber.
\end{lemma}
\begin{proof}
By \Cref{lem:mediandegrees}, every vertex in $X_0$ has at least $m/2$ outneighbours in $A_1\cup A_2$. Then \Cref{lem:mindegturan} with $\beta=1/4$ gives subsets $X'_0\subs X_0$ and $Y \subs A_1\cup A_2$ of size $|X'_0|\ge 2k$ and $|Y|\ge 2m/2^{16k}$ such that $X'_0 \cplt Y$. At least half of $Y$ must lie in the same $A_{i_1}$ for some $i_1=1$ or $i_1=2$, so $Y_{i_1}= Y\cap A_{i_1}$ has size $|Y_{i_1}|\ge m/2^{16k}$.

Similarly, \Cref{lem:mediandegrees} implies that every vertex in $Y_{i_1}$ has at least $m/2$ outneighbours in $A_{i_1+1}\cup A_{i_1+2}$, so $d(Y_{i_1},A_{i_2}) \ge 1/4$ for $i_2=i_1+1$ or $i_2=i_1+2$. We can thus apply \Cref{lem:transturan} with $\beta=1/4$ to find sets $Y'_{i_1}\subs Y_{i_1}$ and $X_{i_2}\subs A_{i_2}$ of size $|Y'_{i_1}| \ge |Y_{i_1}|/2^{64k}\ge m/2^{80k}$ and $|X_{i_2}| \ge 8k$ such that $X_{i_2}$ induces a transitive tournament in $T$, and $Y'_{i_2} \cplt X_{i_2}$.

We will construct the $k$-absorber as follows. We set $S_0=X'_0$, and $S_{r'+1}$ (where $r'=2^{10k}$) will be a $k$-subset of $X_{i_2}$ defined later. We find the sets $S_1,\dots, S_{r'}$ and $Q$ in $Y'_{i_1}$ by applying \Cref{thm:tournamentpath} to obtain a path $v_1v_2 \dots v_{\ell}$ in $Y'_{i_1}$ with $\ell = r'(2k+1) \le m/2^{90k} \le |Y'_{i_1}|/2^{10k}$ vertices, whose $4k$-th power is in the tournament. Let us define $S_i= \{v_{i(2k+1)-1}, \dots , v_{i(2k+1)-2k}\}$ and $q_i= v_{i(2k+1)}$ for every $i=[r']$, and let $Q=\{q_1,\dots, q_{r'}\}$. Then we know that each $S_i$ induces a transitive tournament of size $2k$, and we have $S_0 \cplt Q \cplt S_{r'+1}$ and $S_0\cplt S_1 \cplt \dots \cplt S_{r'+1}$, as well as $S_i\cplt q_i\cplt S_{i+1}$ for every $i\in[r']$.

The crucial part of the construction is closing the cycle. To do so, we apply \Cref{cor:medtranssequence} to the interval $A_{i_2} \cup \dots \cup A_t$ with $X=X_{i_2}$, $X'=X_t$, $F=\emptyset$, and $2k$ in the place of $k$. As $t-i_2\ge 60$, $|X_{i_2}|,|X_t| \ge 8k$, and each $A_i$ has size $m \ge 2^{81000k}$, there are sets $S_{r'+1} \subs X_{i_2}$, $S_{r'+s+1} \subs X_t$ and $S_{r'+2},\dots, S_{r'+s} \subs A_{i_2+1} \cup \dots \cup A_{t-1}$ such that each of $S_{r'+1}, \dots , S_{r'+s+1}$ has size $2k$, induces a transitive subtournament, and $S_{r'+1} \cplt \dots \cplt S_{r'+s+1}$.

By construction, we have $S_{r'+s+1} \cplt S_0$, so we are done with $r=r'+s+1$.
\end{proof}

\section{Proof of the partitioning theorem}\label{secpartition}

\begin{theorem}
Every $n$-vertex tournament $T$ can be covered with at most $2^{10^5k}$ vertex-disjoint $k$-th powers of directed paths.
\end{theorem}
\begin{proof}
Let $H_1,\dots, H_s$ be a maximal collection of vertex-disjoint $k$-absorbers in $T$. By \Cref{lem:absorber1cover}, the vertices $V(H_1)\cup \dots \cup V(H_s)$ can be covered by the $k$-th power of a single directed path $P$. Let $T' = T - (H_1\cup\dots \cup H_s)$ be the subtournament induced by the remaining vertices, and let $n'=|V(T')|$.

Take a median ordering $\prec$ of $T'$, and let us split the vertices into subintervals of size $m=2^{81000k}$. More precisely, we split $V(T')$ into intervals $A_0\prec A_1\prec \dots \prec A_t$ where $t=\floor{n'/m}$, $|A_i|=m$ for $i\in [t]$, and $|A_0| < m$. We can afford to use a single path for each vertex in $A_0$, so let us focus on covering $A_1\cup \dots\cup A_t$.

As $T'$ does not contain any $k$-absorbers, \Cref{lem:absorbfree} tells us that there cannot be indices $i,i'\in [t]$ such that $i'\ge i+80$, and $B' \cplt B$ for some $8k$-sets $B\subs A_i$ and $B'\subs A_{i'}$ that both induce transitive subtournaments. We can therefore apply \Cref{lem:transsequence} to the sets $A_1,A_{81}, A_{161}, \dots$ to find $8k$-subsets $X_1\subs A_1, X_{81} \subs A_{81}, X_{161}\subs A_{161}, \dots$ that each induce transitive subtournaments, and $X_1\cplt X_{81} \cplt X_{161}\cplt \dots$. In fact, we can then apply it again to the sets $A_1\sm X_1, A_{81}\sm X_{81}, A_{161} \sm X_{161}, \dots$ to find another sequence of such $8k$-sets $X'_1\subs A_1, X'_{81} \subs A_{81}, X'_{161}\subs A_{161}, \dots$ disjoint from the $X_i$, and repeat this as long as each $A_i$ contains at least $2^{80k}$ unused vertices.

This way we can find subsets $X^h_i \subs A_i$ of size $8k$ for every $i\in [t]$ and $h = 1,\dots, r$ for some $r$, such that these sets are pairwise disjoint, each of them induces a transitive subtournament, $X^h_{w} \cplt X^h_{w+80} \cplt X^h_{w+160} \cplt \dots$ for every $h\in [r]$ and $w\in [80]$, and $2^{81000k}-8kr < 2^{80k}$, i.e., the set $U_i \subs A_i$ of vertices left uncovered by the $X^h_i$ has size $|U_i|< 2^{80k}$ for every $i\in [r]$. Note that each sequence $X^h_w, X^h_{w+80}, \dots$ contains the $k$-th power of a \emph{spanning} path, and this holds even if we remove at most $7k$ arbitrary vertices from each $X^h_i$. It is therefore enough to cover the remaining vertices $U_1 \cup \dots \cup U_t$ with $k$-th powers of paths, and we can even use some vertices from the $X^h_i$ for this purpose.

Let $m'= 2^{81000k}-8kr < 2^{80k}$ be the size of each $U_i$, and denote its vertices as $U_i = \{u_{i,1},\dots, u_{i,m'}\}$. By \Cref{lem:mediandegrees}, $u_{i,j}$ has at least $79m/2$ outneighbours in the interval $A_{i+1}\cup \dots \cup A_{i+80}$, at least $m/2$ of which must be in $A_{i+40} \cup \dots \cup A_{i+80}$. In particular, at least $m/160$ of these outneighbours must fall in some $A_{\alpha(i,j)}$ with $i+40\le \alpha(i,j)\le i+80$, and similarly, $u_{i,j}$ has at least $m/160$ inneighbours in some $A_{\beta(i,j)}$ with $i-80 \le \beta(i,j) \le i-40$. Note that a given index $\gamma \in [t]$ can appear as $\alpha(i,j)$ or $\beta(i,j)$ for no more than $160m'$ different vertices $u_{i,j}$, so we can choose subsets of in- and outneighbourhoods of size $m/(160^2 m')\ge m/2^{100k}$ so that they are all disjoint. Also, each of these subsets must come from at least $(m/2^{100k})/8k \ge 2^{80000k}$ different $8k$-sets $X^h_{\gamma}$.
All in all, for every $u_{i,j}$, we can find a set $N^+_{i,j} \subs A_{\alpha(i,j)}$ of outneighbours and another subset $N^-_{i,j} \subs A_{\beta(i,j)}$ of inneighbours for $u_{i,j}$ such that these are pairwise disjoint subsets of size at least $2^{80000k}$ each, and they together contain at most one vertex from each $8k$-set $X^h_{\gamma}$.

Recall that there are no indices $i,i'$ with $i'\ge i+80$ such that $B'\cplt B$ for some $8k$-sets $B\subs A_i$ and $B'\subs A_{i'}$. As $\alpha(i,j)\ge \beta(i,j)+80$ and $\beta(i+240,j)\ge \alpha(i,j)+80$ for every $i,j$, this means that we can apply \Cref{lem:transsequence} to the sequence $N^-_{w,j}, N^+_{w,j}, N^-_{w+240,j}, N^+_{w+240,j}, N^-_{w+480,j}, N^+_{w+480,j}, \dots$ to find $8k$-sets $Y^-_{i,j} \subs N^-_{i,j}$ and $Y^+_{i,j} \subs N^+_{i,j}$ that induce transitive tournaments, and satisfy $Y^-_{w,j} \cplt Y^+_{w,j} \cplt Y^-_{w+240,j} \cplt Y^+_{w+240,j} \cplt Y^-_{w+480,j} \cplt Y^+_{w+480,j} \cplt \dots$, for every $w\in[240]$ and $j\in [m']$. As $Y^-_{i,j} \cplt u_{i,j} \cplt Y^+_{i,j}$, we can cover the vertices $Y^-_{w,j} \cplt u_{w,j} \cplt Y^+_{w,j} \cplt Y^-_{w+240,j} \cplt u_{w+240,j} \cplt Y^+_{w+240,j} \cplt Y^-_{w+480,j} \cplt u_{w+480,j}  \cplt Y^+_{w+480,j} \cplt \dots$ with the $k$-th power of a single path.

Putting everything together, we see that we can cover $U_1\cup \dots \cup U_t$ (and the vertices in $Y^{\pm}_{i,j}$) with the $k$-th powers of $240m'$ vertex-disjoint paths. As noted above, these paths use at most one vertex from each $X^h_i$, so we can cover the remaining vertices in $T'$ with the $k$-th powers of $80r$ vertex-disjoint paths, plus one path for each vertex in $A_0$. Together with $P$, we obtain no more than $240m' + 80r + m + 1 \le 2^{90k} + 2^{81010k} + 2^{81000k} + 1 < 2^{10^5 k}$ vertex disjoint paths whose $k$-th powers cover all vertices of $T$.
\end{proof}

\section{Proof of the existence of a long cycle power}\label{seccycle}

First, we show that the bounds in \Cref{thm:cycles} are tight up to the implied constants. 
We can easily construct a roughly $\eps$-\intrans tournament on $n$ vertices with no cycle of length $5\eps n$, as follows. Split the $n$ vertices into $m=1/(4\eps)$ parts of equal size, say $A_1,\dots,A_m$. Let $T_i$ be a random tournament on $A_i$ for every $i$, and orient all $A_i$-$A_j$ edges from $A_i$ to $A_j$ when $i<j$. It is easy to see that a random tournament is roughly $1/4$-\intrans, so $T$ must be roughly $\eps$-\intrans. On the other hand, every cycle intersects at most one of the $T_i$, so its length is at most $4\eps n$. 

We now prove that the bound on the order of the tournament in \Cref{thm:cycles} is tight up to an absolute constant.
\begin{lemma}
For every $k\ge 300$ and $0<\eps<1/8$, there is an $\eps$-\intrans tournament $T$ on at least $\eps^{-k/50}$ vertices that does not contain the $k$-th power of any cycle of length longer than $k$.
\end{lemma}
\begin{proof}
Let $T$ be a transitive tournament on $n=\lceil \eps^{-k/50}\rceil$ vertices, and denote the transitive ordering by $\tau$. Let $R(T)$ be a random tournament obtained by independently reversing each edge with probability $2\eps$.  

First, we show that $R(T)$ is $\eps$-\intrans with probability at least $1/2$. In any fixed ordering $\pi$ of the vertices, there are either at least $n^2/4$ forward edges or $n^2/5$ backward edges. Either way, a standard application of the Chernoff bounds shows that the probability that $R(T)$ contains fewer than $\eps n^2$ backward edges is less than $e^{-\eps n^2/100}$. There are $n!$ orderings, so the probability that $R(T)$ is not $\eps$-\intrans is at most $n! e^{-\eps n^2/100} <e^{n\log n - \eps n^2/100} < 1/2$ using $\eps n>200\log n$ (which is easy to check with the given parameters).

\begin{claim}
If $R(T)$ contains the $k$-th power of a cycle of length at least $k$, then the backward edges of $R(T)$ with respect to $\tau$ contain a copy of $K_{k/10,k/10}$.
\end{claim}
\begin{proof}
Let $C\subs R(T)$ be the $k$-th power of a cycle of length at least $k$. Let $x_1,\dots, x_m$ be an ordering of the vertices of $C$ such that $x_i$ sends edges to $x_{i+1},\dots, x_{i+k}$ for every $i\in [m]$ (and taking indices modulo $m$).

Split $C$ into consecutive intervals of size $k/2$, say $A_1,\dots, A_{\floor{2m/k}}$ (ignoring any leftover vertices), so that $A_1\cplt A_2\cplt \cdots \cplt A_{\floor{2m/k}}\cplt A_1$.
If $A_2$ contains at least $k/10$ vertices that precede some $k/10$ vertices in $A_1$ in $\tau$, then we have the desired $K_{k/10,k/10}$ consisting of backward edges. Hence, we may assume that there are sets $A'_1\subs A_1$ and $A'_2\subs A_2$ of size at least $2k/3$ such that $A'_1 \prec A'_2$. By applying the same argument to $A'_2$ and $A_3$, either we find a $K_{k/10,k/10}$ consisting of backward edges, or there are sets $A''_2\subs A'_2$ and $A'_3\subs A_3$ such that $A''_2$ has size $k/10$, $A'_3$ has size $2k/3$, and $A''_2 \prec A'_3$. Continuing in the same fashion, we obtain a sequence $A'_1 \prec A''_2 \prec \dots \prec A'_{\floor{2m/k}}$. Now $A'_{\floor{2m/k}}\cplt A'_1$ gives a $K_{k/10,k/10}$ of backward edges, as we wanted to show. 
\end{proof}

To finish the proof, it is enough to show that with probability at least $1/2$, there is no complete bipartite $K_{k/10,k/10}$ consisting of backward edges in $\tau$, and hence $R(T)$ does not contains the $k$-th power of any cycle of length at least $k$. There are at most $\binom{n}{k/10}^2$ possible complete bipartite graphs on $n$ vertices, and each of them appears in $R(T)$ with probability $(2\eps)^{k^2/100}$. Now it easy to check that
\[
\binom{n}{k/10}^2 \cdot (2\eps)^{k^2/100}< n^{k/5} \cdot \eps^{k^2/200} = \eps^{-k^2/250+ k^2/200} < 1/2.
\]

This shows that there is an instance of $R(T)$ that is $\eps$-\intrans but does not contain the $k$-th power of any cycle of length at least $k$, as we wanted to show.  
\end{proof}

\medskip
Let us now turn to the proof of \Cref{thm:cycles}. The general idea is the following. We first find two vertex subsets that are relatively far from each other in a median ordering, but are connected by many backward edges. We then use our tools from \Cref{sec:tools} to assemble the $k$-th power of a long cycle as follows. The cycle will start with a complete bipartite graph of backward edges between the two subsets that we can find using \Cref{cor:transturan}. We continue the cycle with forward edges. \Cref{lem:mediansequence} allows us to touch many of the vertices between the two subsets. Finally, we can apply \Cref{cor:medtranssequence} to close the cycle. This method yields a cycle of length $c_k\eps n$. In order to make the constant independent of $k$, we will repeat the above argument several times, ``wrapping around'' the two subsets.

\medskip
Our main tool for finding two sets with many backward edges is the following density-increment lemma, which is inspired by a similar tool of Long \cite[Lemma 5]{L17}.

We define the \emph{length} $\len_\tau(e)$ of an edge $e=v_iv_j$ with respect to a given ordering $\tau = v_1\prec \dots \prec v_n$ of the vertices as the distance $|i-j|$ of its endpoints.

\begin{lemma} \label{lem:densityincrement}
Let $0<\eps\le 1/4$ and let $T$ be a tournament on $n$ vertices that is $\eps$-\intrans. Then for every $0<c<1/3$, $T$ satisfies at least one of the following properties. 
\begin{enumerate}[(P1)]
    \item \label{itm:p1} In every median ordering $\tau$ of $T$, there is a set $E_{\tau}$ consisting of backward edges such that $|E_{\tau}|\ge  c\eps n^2/4$ and $\len_{\tau}(e)\geq cn/4$, for every $e\in E_{\tau}$.
    \item \label{itm:p2} There is a sub-tournament $T'\subs T$ of order at least $n/2$ such that $T'$ is $2(1-c)\eps$-\intrans.
\end{enumerate}
\end{lemma}
\begin{proof}
Let $\tau = v_1\prec \dots \prec v_n$ be a median ordering of $T$. By assumption, there are at least $\eps n^2$ backward edges in this ordering. Let $E_{\tau}$ be the set of backward edges whose endpoints are at distance at least $cn/4$. If $|E_{\tau}|\ge \frac{c\eps}{4} n^2$, then \ref{itm:p1} is satisfied, so we may assume that $|E_{\tau}| < \frac{c\eps}{4} n^2$. In particular, we may assume that $cn/4>1$ (i.e., $n>4/c$), as otherwise $E_{\tau}$ contains at least $\eps n^2$ backward edges.

Now let $I_1 = \{v_1,\dots, v_{\ceil{n/2}}\}$ and $I_2 = \{v_{\floor{n/2}+1}, \dots, v_n\}$, and let $E_1$ and $E_2$ be the respective sets of backward edges induced by them. We may assume that $|E_1|\ge |E_2|$. Finally, let $F$ be the set of backward edges not contained in either of $E_1,E_2$ and $E_{\tau}$. Then the edges in $F$ must go from $I_2$ to $I_1$ and have length less than $cn/4$, so they are all induced by the interval $J = \{v_1,\dots, v_{\floor{(1+c/2)n/2}}\}$.

We have $|E_1|+|E_2|+|E_{\tau}|+|F| = \eps n^2$ with $|E_1|\ge |E_2|$ and $|E_{\tau}| < \frac{c\eps}{4}n^2$. If $|F| < \frac{c\eps}{4}n^2$, then $|E_1| > \frac{\eps(1-c/2)}{2}n^2$. Using $n>4/c$, it is easy to check that $\frac{cn^2}{2} \ge (1-c)(2n+1)$, so we have $|E_1|\ge  \frac{\eps(1-c)}{2}(n+1)^2 \ge 2(1-c)\eps|I_1|^2$. As $\tau$ is a median ordering of the subtournament induced by any interval, we can choose $T'=T[I_1]$ to satisfy \ref{itm:p2}. Otherwise, $|F| \ge \frac{c\eps}{4}n^2$, so $|E_1|+|F| \ge \eps n^2/2 \ge (1-c)(1+c/2)^2\eps n^2/2 \ge 2(1-c)\eps |J|^2$. We can then choose $T'=T[J]$.
\end{proof}

\begin{corollary} \label{cor:longbackedges}
Let $0<\eps \le 1/4$, and let $T$ be an $n$-vertex tournament that is $\eps$-\intrans. Then, for some $\teps \ge\eps$, $T$ contains a subtournament $\tT$ on $\tn$ vertices that is $\teps$-\intrans such that $\teps\tn\ge \eps n/5$, and for every median ordering $\tau$ of $\tT$ there is a set $E_{\tau}$ consisting of $\teps^2 \tn^2/4$ backward edges such that $\len_{\tau}(e)\ge \teps\tn/4$, for every $e\in E_{\tau}$.
\end{corollary}
\begin{proof}
Let us repeatedly apply \Cref{lem:densityincrement} with $c=\eps$ as long as \ref{itm:p2} holds, i.e., let $T=T_0 \sups T_1 \sups \dots \sups T_p$ be a longest sequence of tournaments such that $n_i = |V(T_i)| \ge |V(T_{i-1})|/2$ and $T_i$ is $\eps_i$-\intrans for some $\eps_i\ge 2(1-\eps_{i-1})\eps_{i-1}$, for every $i\in [p]$. We claim that $\tT=T_p$ satisfies the conditions.

To see this, first note that if $T_i$ is $\eps_i$-\intrans, then $\eps_i \le 1/4$. Since $\eps_{i-1} \le \frac{2}{3} \eps_i$ for every $i\in[p]$, we have $\sum_{i=0}^p \eps_i \le \eps_p \sum_{i=0}^p (2/3)^i \le 3/4$.
We also know that $n_p \ge n_{p-1}/2 \ge \dots \ge n_0/2^p$ and that 
\[
1/4 \ge \eps_p \ge 2\eps_{p-1}(1-\eps_{p-1}) \ge \dots \ge 2^p \eps_0 \prod_{i=0}^{p-1} (1-\eps_i) \ge 2^p \eps e^{- 2\sum \eps_i} \ge 2^p\eps/e^{3/2},
\]
so $\teps\tn \ge \eps n/e^{3/2} > \eps n/5$ for $\teps=\eps_p$ and $\tn=n_p$. 

As we chose a maximal sequence of tournaments, $\tT$ does not satisfy \ref{itm:p2} with $c=\teps$. But then \Cref{lem:densityincrement} implies \ref{itm:p1} for $\tT$ with $c=\teps$, which is exactly what we wanted to show.
\end{proof}

We are now ready to prove our theorem.

\begin{theorem}
Every $\eps$-\intrans tournament on $n\ge \eps^{-41000k}$ vertices contains the $k$-th power of a cycle of length at least $\eps n/1500$.
\end{theorem}
\begin{proof}
Let us apply \Cref{cor:longbackedges} to $T$, and let $\tT$ be the $\teps$-\intrans $\tn$-vertex subtournament we obtain, where $\teps\tn \ge \eps n/5$ and $1/4\ge \teps\ge \eps$. Note that this implies $\teps\tn\ge \eps^{-40950k}\ge \teps^{-40950k}$.

Fix any median ordering $\tau$ of $\tT$. Then there is a set $E_{\tau}$ of $\teps ^2\tn^2/4$ backward edges in $\tT$ such that every edge $e\in E_{\tau}$ has length $\len_{\tau}(e) \ge \teps\tn/4$.

Let us split the $\tn$ vertices of $\tT$ into $t=12/\teps$ intervals $A_1\prec \dots \prec A_t$ of size $m=\tn/t =\teps\tn/12 \ge \teps^{-40900k}$ each. Then every edge of $E_{\tau}$ must connect two intervals $A_i$ and $A_{i'}$ with at least 2 other intervals in between (i.e., $i'\ge i+3$). In particular, there must be at least $|E_{\tau}| / t^2 \ge \teps^2 m^2/200$ backward edges going from $A_b$ to $A_a$ for some $a,b$ satisfying $b\ge a+3$.
Now, split each $A_i$ into $t'=m/\teps^{20400k}\ge \teps^{-20500k}$ consecutive subintervals $A_{i,1}\prec \dots\prec A_{i,t'}$ of size $m'=\teps^{-20400k}$. 

\begin{claim}
There are at least $\frac{\teps^2t'}{800}$ disjoint pairs $\{A_{a,j}, A_{b,j'}\}$ such that $d(A_{b,j'}, A_{a,j})\ge \frac{\teps^2}{400}$.
\end{claim}
\begin{proof}
Note that $\sum_{j,j'} d(A_{b,j'},A_{a,j}) = t'^2\cdot d(A_b,A_a)\ge \teps^2t'^2/200$. This means that at least $\teps^2t'^2/400$ of the pairs satisfy $d(A_{b,j'},A_{a,j})\ge \teps^2/400$. We can then greedily choose $\teps^2t'/800$ of these pairs so that they are disjoint from each other.
\end{proof}

Let $r=\frac{m'}{50k}\le \teps^{-20400k}\le \frac{\teps^2t'}{800}$ and let $\{A_{a,j_1}, A_{b,j'_1}\}, \dots, \{A_{a,j_r}, A_{b,j'_r}\}$ be disjoint pairs provided by the claim. We may assume that $j_1>\dots>j_r$.

As $m'>(\frac{\teps}{400})^{-160k}$, we can apply \Cref{cor:transturan} to find sets $Z_i\subs A_{a,j_i}$ and $Z'_i\subs A_{b,j'_i}$, for every $i\in[r]$, that induce transitive tournaments of size $32k$ and satisfy $Z'_i\cplt Z_i$.

Next, we apply \Cref{lem:mediansequence} one by one for every $i\in[r]$ to the interval $A_{a,j_i}\cup A_{a,j_i+1}\cup\dots \cup A_{a+1,t'}$ with $X=Z_i$ to find $4k$-sets $X_{i,1}, \dots, X_{i,s_i}$ such that each of them induces a transitive tournament, and they satisfy $X_{i,1}\subs Z_i$ and $X_{i,1}\cplt\cdots\cplt X_{i,s_i}$. The lemma also implies that for every $j\in[t'-1]$, one of these sets is contained in $A_{a+1,j}\cup A_{a+1,j+1}$ (for example, $X_{i,s_i}\subs A_{a+1,t'-1}\cup A_{a+1,t'}$). In particular, $X_{i,1}\cup\dots\cup X_{i,s_i}$ contains at least $2kt'$ vertices from $A_{a+1}$.

Moreover, by defining the set of forbidden vertices in the lemma as $F=\bigcup_{i'<i} (X_{i',1}\cup\dots\cup X_{i',s_{i'}})$ (which is allowed because $|F\cap A_{\alpha,j}| \le 4kr\le \frac{m'}{8}$ for every $\alpha\in\{a,a+1\}$ and $j\in[t']$), we can ensure that the path blowups $X_{i,1}\cplt\cdots\cplt X_{i,s_i}$ are vertex-disjoint over $i\in[r]$. Combining this with the fact that each such path blowup contains at least $2kt'$ vertices, we see that the $r$ path blowups together cover at least $2kt'r = \frac{t'm'}{25} = \frac{m}{25} = \frac{\teps\tn}{300} \ge \frac{\eps n}{1500}$ vertices.

All we are left to do is connect the path blowups into one big cycle power. We use \Cref{cor:medtranssequence} with $X=X_{i,s_i}$ and $X'=Z'_{i+1}$ for every $i\in[r]$. More precisely, let $t_i\in \{t-1,t\}$ be the index such that $X=X_{i,s_i}\in A_{a+1,t_i}$, and let $X'$ be any $4k$-set in $Z'_{i+1}$ (or $Z'_1$ if $i=r$). We will apply \Cref{cor:medtranssequence} to the interval $A_{a+1,t_i}\cup \dots \cup A_{b,j'_{i+1}}$ to get $k$-sets $Y_{i,0}\cplt\cdots\cplt Y_{i,s'_i}$ such that $Y_{i,0}\subs X_{i,s_i}$ and $Y_{i,s'_i}\subs Z'_{i+1}$ and each of these sets induces a transitive subtournament. This is possible because the interval (which contains the entire $A_{a+2}$) is split into at least $t'>60$ subintervals of size $m'\ge 2^{40500k}$.

To make the path blowups $Y_{i,0}\cplt\cdots\cplt Y_{i,s'_i}$ vertex-disjoint over $i\in[r]$, we just need to forbid all vertices $F=\left(\bigcup_{i'\in[r]} X_{i',s_{i'}}\right) \cup \left(\bigcup_{i'\in[r]} Z'_{i'}\right) \cup \left(\bigcup_{i'<i} (Y_{i',0}\cup\dots\cup Y_{i',s'_{i'}})\right) $ that are already used by other path blowups. As $s'_i\le 5$ for every $i\in[r]$, this set contains $|F|\le 4kr + 8kr + 6kr =18kr \le m'/2$ vertices, so can indeed be used in our applications of \Cref{cor:medtranssequence}.

All in all, we found disjoint vertex sets $X_{i,j}$ and $Y_{i,j}$ that each induce transitive subtournaments of size at least $k$, together contain at least $\frac{\eps n}{1500}$ vertices, and satisfy
\[
X_{1,1}\cplt\cdots\cplt X_{1,s_1-1} \cplt Y_{1,0}\cplt\cdots\cplt Y_{1,s'_1}\cplt X_{2,1}\cplt\cdots \cplt Y_{r,s'_r}\cplt X_{1,1}
\]
This sequence contains the blowup of a cycle of length $\frac{\eps n}{1500}$, as needed.
\end{proof}


\begin{thebibliography}{9}

\bibitem{BT10}	
S. Bessy and S. Thomass\'e,
\emph{Partitioning a graph into a cycle and an anticycle, a proof of Lehel's conjecture,}
J. Combin. Theory Ser. B \textbf{100} (2010), 176--180.

\bibitem{BH90}	
B. Bollob\'as and R. H\"aggkvist,
\emph{Powers of Hamilton cycles in tournaments,}
J. Combin. Theory Ser. B \textbf{50} (1990), 309--318.

\bibitem{BCFPS20}
S. Bustamante, J. Corsten, N. Frankl, A. Pokrovskiy, and J. Skokan,
\emph{Partitioning edge-coloured hypergraphs into few monochromatic tight cycles,}
SIAM J. Discrete Math. \textbf{34} (2020), 1460--1471.

\bibitem{CG91}
F.R.K Chung and, R. L. Graham,
\emph{Quasi-random tournaments,}
 J. Graph Theory \textbf{15} (1991), 173--198.

\bibitem{DDF+}
N. Dragani\'c, F. Dross, J. Fox, A. Gir\~ao, F. Havet, D. Kor\'andi, W. Lochett, D. Munh\'a Correia, A. Scott, and B. Sudakov,
\emph{Powers of paths in tournaments,}
Combin. Probab. Comput., to appear.

\bibitem{DMS}
N. Dragani\'c, D. Munh\'a Correia, and B. Sudakov,
\emph{Tight bounds for powers of Hamilton cycles in tournaments,}
(2021) arXiv:2103.10414 preprint.

\bibitem{EGP91}
P. Erd\H{o}s, A. Gy\'arf\'as, and L. Pyber,
\emph{Vertex coverings by monochromatic cycles and trees,}
J. Combin. Theory Ser. B \textbf{51} (1991), 90--95.

\bibitem{F21}
J. Fox, X. He, and Y. Widgerson,
\emph{Ramsey numbers of sparse digraphs,}
(2021) arXiv:2105.02383 preprint.

\bibitem{FS08}
J. Fox, B. Sudakov,
\emph{Unavoidable patterns,}
J. Combin. Theory, Series A \textbf{115} (2008), 1561--1569. 

\bibitem{G16}
A. Gy\'arf\'as,
\emph{Vertex covers by monochromatic pieces—a survey of results and problems,}
Discrete Math. \textbf{339} (2016), 1970--1977. 

\bibitem{L17}
E. Long,
\emph{Large unavoidable subtournaments,}
Combin. Probab. Comput. \textbf{26} (2017), 68--77.

\bibitem{LRS98}
T. \L uczak, V. R\"{o}dl, E. Szemer\'edi,
\emph{Partitioning two-colored complete graphs into two monochromatic cycles,}
Combin. Probab. Comput. \textbf{7} (1998), 423--436.

\bibitem{R34}
L. R\'edei,
\emph{Ein kombinatorischer Satz}
Acta Litt. Sci. Szeged \textbf{7} (1934), 39--43.

\bibitem{Y21}
R. Yuster,
\emph{Paths with many shortcuts in tournaments,}
Discrete Math. \textbf{334} (2021), 112168.

\end{thebibliography}
\end{document}